\documentclass[12pt,a4paper,oneside]{amsart}

\usepackage{amsfonts, amsmath, amssymb, amsthm, hyperref}
\usepackage{anysize}

\newtheorem{theorem}{Theorem}[section]
\newtheorem*{theorem*}{Theorem}
\newtheorem{lemma}[theorem]{Lemma}

\newtheorem{corollary}[theorem]{Corollary}

\theoremstyle{definition}
\newtheorem{definition}[theorem]{Definition}

\theoremstyle{remark}
\newtheorem{remark}[theorem]{Remark}

\numberwithin{equation}{section}

\renewcommand{\epsilon}{\varepsilon}
\renewcommand{\phi}{\varphi}
\renewcommand{\kappa}{\varkappa}

\begin{document}

\title{Size of components of a cube coloring}
\author{Marsel~Matdinov}
\email{marmarselsel@mail.ru}
\address{Marsel~Matdinov, Faculty of Mathematics, Higher School of Economics, 7 Vavilova Str. Moscow, Russia 117312}


\begin{abstract}
Suppose a $d$-dimensional lattice cube of size $n^d$ is colored in several colors so that no face of its triangulation (subdivision of the standard partition into $n^d$ small cubes) is colored in $m+2$ colors. Then one color is used at least $f(d, m) n^{d-m}$ times.
\end{abstract}

\maketitle

\section{Introduction}

A theorem attributed to Lebesgue asserts that if a lattice cube of dimension $d$ is colored in $d$ colors then one of the colors has a connected component spanning two opposite facets of the cube. By the standard reasoning with nerves of coverings this means that the covering dimension of the $d$-dimensional cube is at least $d$.

There arises the following natural question: What happens if the number of colors is less than $d$? In~\cite{matpri2008,kb2010} it was conjectured that the size of a monochromatic connected component has a lower bound of order $n^{d-m}$ when $m+1$ colors are used. For $m=d-1$ this follows from the Lebesgue theorem and for $m=1$ this conjecture is proved in~\cite{matpri2008}.

Here we prove this conjecture in a slightly stronger form:

\begin{theorem}
\label{color-cube}
Let a $d$-dimensional cube $Q$ be partitioned into $n^d$ small cubes in the standard way and then the $(m+1)$-dimensional skeleton $Q_m$ of this partition is subdivided to the triangulation $T$. Suppose the vertices of $T$ (equivalently, vertices of $Q_m$) are colored in several colors so that no $(m+1)$-face $\sigma\in T$ is colored in $m+2$ different colors. Then one of the colors is used on at least $f(d, m) n^{d-m}$ vertices of $T$.
\end{theorem}

\begin{remark}
It is also sufficient to assume that every cubical face of $Q_m$ of dimension $m+1$ has at most $m+1$ colors on its vertices. Such a point of view allows not to use any triangulation $T$ in the statement of the theorem.
\end{remark}

\begin{remark}
\label{color-comp}
If a color $c$ has several connected components then every component can be assumed to be a separate color. So we obtain a monochromatic connected set of at least $f(d,m) n^{d-m}$ vertices of $T$. By \emph{adjacent} vertices we mean two vertices in a single $(m+1)$-face of $T$. This remark also remains valid if we define adjacent vertices as contained in a single cubical $(m+1)$-face of $Q_m$.
\end{remark}

\begin{remark}
We do not establish any explicit values for $f(d,m)$. The reader may consult the paper~\cite{kar2011} for this information.
\end{remark}

A similar theorem was proved independently in~\cite{kar2011}. In Section~\ref{corollaries-sec} we give some corollaries of Theorem~\ref{color-cube} about coverings of a cube or a torus.

\medskip
\textbf{Acknowledgments. }
The author thanks Alexey~Kanel-Belov for drawing attention to this problem and numerous discussions; and thanks Roman~Karasev for his help in writing the text and translating it to English.

\section{Proof of Theorem~\ref{color-cube}}

Let us introduce some notation. Let $C = \{c_0, \ldots, c_k\}$ be an ordered set of $k+1$ colors. For every oriented $k$-face of $T$ we assign $+1$ is its vertices are colored in the colors of $C$ in accordance with the orientation, $-1$ if the vertices are colored in the colors of $C$ with opposite orientation, and $0$ if the face is colored in other colors or some of the colors is used more than once. Thus we define a cochain $\chi(C) \in C^k(T)$ and observe the coboundary formula:
\begin{equation}
\label{coboundary-eq}
\delta \chi(C) = \sum_{c_{k+1}} \chi_{Cc_{k+1}},
\end{equation}
where we sum over all the colors and by $Cc_{k+1}$ we mean the concatenation of $C$ and the new color $c_{k+1}$. If $c_{k+1}$ coincides with a color in $C$ then $\chi_{Cc_{k+1}}$ is assumed to be zero.

In the rest of the proof we consider $(k+1)$-dimensional cubical subcomplexes $Q_k\subseteq Q_m$, which are $(k+1)$-dimensional cubical skeleta of some $(d-m+k)$-dimensional faces (big faces, not faces of a partition) of the cube $Q$, for $k=m,m-1,\ldots, 0$. 

In order to apply simplicial cochains to cubical chains we introduce the function $L \colon C_k(Q_m) \to C_k(T)$ that assigns to any $k$-face  $\tau\in Q_m$ the sum of simplicial $k$-faces of $T$ that partition $\tau$ with appropriate orientations. Obviously $L$ commutes with the boundary map $\partial$.

We are going to \emph{balance} the complexes $Q_k$ as follows:

\begin{definition}
For every $k$-face $\sigma\in Q_k$ we will assign a $(k+1)$-chain $B(\sigma)\in C_{k+1}(T)$ so that for any $(k+1)$-face $\tau\in Q_k$ and a set $C$ of $k+1$ colors the following holds:
\begin{equation}
\label{balancing-eq}
\chi_C (L(\partial \tau) + \partial B(\partial \tau)) = 0,
\end{equation}
where we assume that $B$ is linearly extended to $k$-chains of $Q_k$.
\end{definition}

We also put $A(\sigma) = L(\sigma) + \partial B(\sigma)$ and use (\ref{balancing-eq}) in the following form:
\begin{equation}
\label{balancing-eq-2}
\chi_C (A(\partial \tau)) = 0.
\end{equation}

Let us check that the $(m+1)$-skeleton $Q_m$ of $Q$ is already balanced (so that we may put $B(\cdot)=0$). Let $|C|=m+1$, then 
$$
\chi_C(L(\partial \tau)) = (\delta \chi_C, L(\tau)) = 0,
$$ 
since $\delta \chi_C=0$ by the formula (\ref{coboundary-eq}) and the assumption of the theorem (no face of $T$ is colored in $m+2$ different colors).

The plan of the remaining part of the proof is following:

\begin{itemize}
\item 
Denote the $k$-skeleton of any facet of the cube corresponding to $Q_k$ by $Q_{k-1}$;
\item 
Balance $Q_{k-1}$ by defining a suitable $B : C_{k-1}(Q_{k-1})\to C_k(T)$;
\item 
Make sure that in all the expressions $B(\sigma)$ for all $(k-1)$-faces $\sigma\in Q_{k-1}$ every $k$-face $\beta\in T$ is used at most $C(d, k-1)$ times (counted with its multiplicity in the chains $B(\sigma)$).
\end{itemize}

If this plan passes then on the last stage we have a $1$-dimensional skeleton of a $(d-m)$-dimensional cube (containing $n^{d-m}$ small cubes) $Q_0$. To every vertex $v\in Q_0$ we assign a chain of $1$-faces of $T$ denoted by $B(v)$. Then $0$-chains $A(v)$ are simply sets of vertices of $T$ with integer multiplicities such that the sum of coefficients in every $A(v)$ is $1$. For any color $c$ by (\ref{balancing-eq-2}) we obtain that $\chi_c (A(v_1)) = \chi_c(A(v_2))$ for any pair of adjacent vertices $v_1$ and $v_2$. Since $Q_0$ is a connected graph we obtain that the number $\chi_c(A(v)) = x_c$ does not depend on $v$. The sum over all colors is 
$$
\sum \chi_c(A(v)) = (1, A(v)) = 1,
$$ 
so there exists a color $c$ with nonzero $x_c$. Hence this color is used in every support of the $0$-cycle $A(v)$ for every $v$. We have at least $n^{d-m}$ different choices of $v$ and any point colored in $c$ is counted at most $C' C(d, 0)$ times (here $C'$ is the maximal number of $1$-faces incident to a vertex in $T$).

So it remains to pass from the balancing of $Q_k$ to the balancing of $Q_{k-1}$. Note that for every $k$-face $\tau\in Q_{k-1}$ we have to satisfy the equality (since $\partial^2 = 0$ and $\partial A(\tau) = \partial L(\tau) = L(\partial \tau)$):
\begin{equation}
\label{balancing-eq-3}
\chi_C (\partial A(\tau) + \partial B(\partial \tau)) = \chi_C (L(\partial \tau) + \partial B(\partial \tau)) = 0.
\end{equation}
In this formula $A(\tau)$ is already defined, and $B$ is to be defined on  $(k-1)$-faces of $Q_{k-1}$. The equality (\ref{balancing-eq-3}) follows from the equality:
\begin{equation}
\label{balancing-eq-4}
\chi_D (A(\tau) + B(\partial \tau)) = 0
\end{equation}
for every $k$-face $\tau\in Q_{k-1}$ and every set $D$ of $k+1$ colors. Indeed, using (\ref{coboundary-eq}) from (\ref{balancing-eq-4}) we obtain:
\begin{multline}
\chi_C (\partial A(\tau) + \partial B(\partial \tau)) = (\delta \chi_C, A(\tau) + B(\partial \tau) ) = \\
= (\sum_{c_k} \chi_{Cc_k}, A(\tau) + B(\partial \tau)) = 0.
\end{multline}

Now we fix a set $D$ of $k+1$ colors. Define by 
\begin{equation}
\xi_D(\tau) = \chi_D(A(\tau))
\end{equation}
a $k$-cocycle on $Q_k$ since for every $(k+1)$-face $\rho\in Q_k$ we have:
\begin{equation}
\xi_D(\partial \rho) = \chi_D(L(\partial \rho) + \partial B \partial \rho) = 0
\end{equation}
because $Q_k$ is balanced. To make this cocycle zero (as required in (\ref{balancing-eq-4})) we have to assign to some $(k-1)$-faces $\sigma\in Q_{k-1}$ as $B(\sigma)$ some sets (with coefficients) of $k$-faces $\tau\in T$. Obviously, it suffices to use only those $k$-faces $\tau\in T$ that are colored exactly in the colors of $D$.

The map $\sigma\mapsto \xi_D(B(\sigma))$ is going to be a $(k-1)$-dimensional cochain in $C^{k-1}(Q_{k-1})$ with coboundary $\xi_D|_{Q_{k-1}}$. In order to use any $k$-face (out of those colored in $D$) at most $C(d, k-1)$ times we have to check that the ratio between the norm (sum of absolute values) of some $(k-1)$-cochain $\eta\in C^{k-1}(Q_{k-1})$ such that $\delta \eta = \xi_D$ on $Q_{k-1}$ and the number of $k$-faces usable in $B(\sigma)$ (that is, colored in $D$) is bounded by a constant $C(d, k-1)$.

Let the norm of $\xi_D$ as an element of $C^k(Q_k)$ equal $M$. By the assumption that in every $B(\tau)$ a $(k+1)$-face of $T$ is used at most $C(d, k)$ times we conclude that every $k$-face of $T$ is used in all $A(\tau)$ at most $C(d,k)C'(k)$ times, where $C'(k)$ is the maximal number of $(k+1)$-faces containing a given $k$-face of $T$ (it can be bounded independently on the choice of a particular triangulation $T$). By the formula $\xi_D(\tau) = \chi_D(A(\tau))$ we conclude that among the $k$-faces of $T$ there do exist at least $\frac{M}{C(d,k)C'(k)}$ ``candidates'' for $B(\sigma)$. Now it suffices to solve the equation $\delta \eta = \xi_D$ on cochains on $Q_{k-1}$ so that the norm $|\eta|$ is at most $C''(d, k-1)M$. After that we can assign to cubical faces of $Q_{k-1}$ on which $\eta$ is nonzero several faces of $T$ on which $\xi_D$ is nonzero.

Note that $|\xi_D|=M$ and for some codimension $1$ cubical section $Q'$ of $Q_k$ (parallel to $Q_{k-1}$) we have: $|\xi_D|_{Q'}| \le M/n$. Then we use the ``filling inequality'' (see for example~\cite{grom2010}, where filling inequalities are widely used):

\begin{lemma}
\label{isoperimetry}
For a $k$-dimensional cocycle $\alpha$ on the cubical partition of the $d'$-dimensional cube $Q'$ (in terms of this proof) there exists a $(k-1)$-dimensional cubical cochain $\beta$ such that $\delta \beta = \alpha$ and $|\beta| \le C_F(d', k) n |\alpha|$.
\end{lemma}

By this lemma we select a $(k-1)$-dimensional cochain $\beta$ on $Q'$ with norm at most $M C_F(d-m+k, k)$ with coboundary $\xi_D|_{Q'}$. Denote the part of $\xi_D$ between $Q'$ and $Q_{k-1}$ by $\xi'_D$; this is a cochain with norm at most $M$. As the required $(k-1)$-dimensional cochain $\eta$ on $Q_{k-1}$ we may take:
\begin{equation}
\eta = \beta + \pi_*(\xi'_D)
\end{equation}
with norm at most $(C_F(d-m+k, k) + 1)M$. Here $\beta$ is moved from $Q'$ to $Q_{k-1}$ by the translation and by $\pi_*(\xi'_D)$ we mean the \emph{direct image} under the projection onto $Q_{k-1}$ that drops the dimension by $1$. The cochain $\pi_*(\xi'_D)$ can be defined explicitly (thanks to the cubical complexes that we use) as taking any $(k-1)$-face $\sigma\in Q_{k-1}$ to the sum of values (with appropriate signs) of $\xi'_D$ on $k$-faces of $Q_k$ that project onto $\sigma$. In other words:
$$
\pi_*(\xi'_D)(\tau) = \xi'_D(\pi^{-1}(\tau)).
$$

So we satisfy the equality (\ref{balancing-eq-4}) for a particular color set $D$. Now we can add the chains $B(\sigma)$ corresponding to different sets $D$. The expressions $B(\sigma)$ for every particular $D$ contained exclusively $k$-faces colored in the colors in $D$. Hence for different $D$ we use different faces, which guarantees the bounded multiplicity of the union of all $B(\sigma)$. The faces of $B(\sigma)$ colored in $D$ do not affect the equality (\ref{balancing-eq-4}) for another set $D'$ (not obtained from $D$ by a permutation). Now to complete the proof it remains to prove the lemma.

\begin{proof}[Proof of Lemma~\ref{isoperimetry}] 
The proof is similar to the proof of~\cite[Lemma~2.6]{kar2011}, which is stated in terms of cycles Poincar\'e dual to the cocyles in this proof.

Put $\alpha_0 = \alpha$. We are going to build a cocycle $\alpha_{i+1}$ out of $\alpha_i$ as follows. Take a hyperplane section $Z$ of $Q'$ parallel to a pair of its opposite facets so that $|\alpha_i|_Z| \le |\alpha_i|/n$. This is possible by the Dirichlet principle.

Define a $(k-1)$-cochain $\beta_i$ as follows: 
$$
\beta_i(\tau) = \alpha_i([\tau, \pi_Z(\tau)]),
$$
where $[\tau, \pi_Z(\tau)]$ is (at most $k$)-dimensional parallelepiped between $\tau$ and its projection onto $Z$ with appropriate sign.
 
Now we put 
$$
\alpha_{i+1} = \alpha_i - \delta\beta_i
$$ 
and note that $|\beta_i|$ is at most $n|\alpha_i|$. Note also that $\alpha_{i+1}$ takes the same values as the translation of $\alpha_i|_Z$ on sections parallel to $Z$ and is zero on any face orthogonal to $Z$.

After several such operations for different directions of $Z$ ($d'-k+1$ will be enough) the cocycle $\alpha_{i+1}$ becomes zero. We have the inequality:
$$
|\alpha_{i+1}| \le |\alpha_i| \le\dots \le |\alpha|.
$$
If we take $\beta =\sum_i \beta_i$ then the required inequality holds with constant $C_F(d',k) \le d'-k+1$.
\end{proof}

\section{Some corollaries}
\label{corollaries-sec}
We give some topological corollaries of Theorem~\ref{color-cube}:

\begin{corollary}\footnote{The statement of this corollary is suggested by R.~Karasev as a simpler version of Corollary~\ref{1-hom-rank}.}
\label{touch-opposite}
Let a $d$-dimensional cube $Q$ be covered by closed sets $C_i$ so that no point is covered more than $m+1$ times. Then one of the sets $C_i$ intersects at least $d-m$ pairs of opposite facets of $Q$.
\end{corollary}

\begin{remark}
This is a generalization of the Lebesgue theorem.
\end{remark}

\begin{proof}
We pass in a standard way from the covering to coloring the vertices of the partition of $Q$ into $n^d$ cubes. If the partition is fine enough then no partition face has $m+2$ distinct colors. 

Let us include $Q$ into the cube $2Q$ of size $(2n)^d$ and repeat the coloring of $Q$ using reflections with respect to the halving hyperplanes of $2Q$. Then we extend the coloring onto the whole $\mathbb Z^d$ with translations by $\pm 2n$ along the coordinate axes. Let us see what happens with a color $c_i$. Following Remark~\ref{color-comp} we assume that the color $c_i$ makes a connected subset of $Q$. The vertices of $\mathbb Z^d$ colored in $c_i$ can be decomposed into connected components; denote one of them by $c'_i$. If the component $c_i$ spans a pair of opposite facets (orthogonal to a base vector $e_j$) in $Q$ then $c'_i$ is invariant under the translation by $\pm 2ne_j$. Otherwise $c_i$ does not touch one of the facets orthogonal to $e_j$ and $c'_i$ is trapped between a pair of hyperplanes orthogonal to $e_j$ at distance $2n$ from each other. 

So the free Abelian group $\Lambda_i$ of translational symmetries of $c'_i$ has dimension exactly $\ell$, where $\ell$ is the number of pairs of opposite facets of $Q$ intersected by $c_i$ and $c'_i$ can be obtained from $2Q\cap c'_i$ by translations in $\Lambda_i$. If we intersect $c'_i$ with a large cube $Q'$ of size $(2nN)^d$ then the cardinality of $c'_i\cap Q'$ has the growth order $N^\ell$ for varying $N$. By Theorem~\ref{color-cube} and Remark~\ref{color-comp} some $c'_i\cap Q'$ must have the number of vertices of order at least $N^{d-m}$; so for some of $c'_i$ we must have $\ell\ge d-m$.
\end{proof}

\begin{corollary}
\label{1-hom-rank}
Let a $d$-dimensional torus $T^d$ be covered by open sets $C_i$ so that no point is covered more than $m+1$ times. Then for some $C_i$ the image of $H_1(C_i)$ in $H_1(T^d)=\mathbb Z^d$ has dimension at least $d-m$.
\end{corollary}

\begin{remark}
The sets have to be open so that the connectedness and the arcwise connectedness coincide.
\end{remark}

\begin{proof}
As in the previous proof we pass from the covering of $T^d$ to a fine enough triangulation of a covering cube $Q$, which subdivides the cubical partition into $n^d$ small cubes. Then we assume that the vertices of the triangulation are colored so that no face has more than $m+1$ colors. Duplicating $Q$ by translations we obtain a large cube $Q_N$ with side length $Nn$ and the corresponding coloring. By gluing the opposite facets of $Q_N$ we obtain a torus naturally $N^d$-fold covering $T^d$. 

By Theorem~\ref{color-cube} in $Q_N$ we have a monochromatic connected component $S$ with size of order $N^{d-m}$. Let $S_1$ be the maximal intersection of $S$ with a residue class modulo $Q$ (points are equal modulo $Q$ if the differences of their coordinates are divisible by $n$). Then $|S_1| \ge |S|/{n^d} \ge Ñ\frac{N^{d-m}}{n^d}$.

Note that a projection of a monochromatic path in $S$ starting in a point of $S_1$ with coordinates $(x_1, \ldots, x_d)$ and ending in a point of $S_1$ with coordinates $(y_1,\ldots, y_d)$ is a monochromatic closed loop in $T^d$ representing the homology class $(\frac{y_i-x_i}{n})_i$. So it suffices to show that the dimension of the linear space generated by pairwise differences of $S_1$ is at least $d-m$ (the covering set corresponding to $S$ will be the one required). Equivalently, we have to show that the dimension of the affine hull of  $S_1$ is at least $d-m$.

Assume the contrary: The dimension of the affine hull of $S_1$ is at most $d-m-1$. Then $S_1$ is contained in at most $(d-m-1)$-dimensional affine subspace $L$ and the number of vertices in $Q_N\cap L$ (and therefore the number of vertices in $S_1$) is at most $(Nn)^{d-m-1}$. For large enough $N$ we obtain a contradiction with the inequality $|S_1| \ge Ñ\frac{N^{d-m}}{n^d}$.
\end{proof}

\end{document}